\documentclass[final,leqno,letterpaper]{etna}
\setbibdata{1}{xx}{46}{2017} 
\hypersetup{%
    pdftitle={A Jacobi-Type  Eigensolver for   Diagonally Dominant Symmetric  Matrices},
    pdfauthor={Luca Gemignani},
    pdfkeywords={Eigenvalue computation,  Jacobi method, diagonal dominance }
    }
\usepackage{graphicx} 
\usepackage{siunitx}
\usepackage{multirow}
\usepackage{enumitem}
\usepackage{lipsum}
\usepackage[english]{babel} 
\usepackage{blindtext}
\usepackage[caption=false]{subfig}
\usepackage{hyperref}
\usepackage{derivative}
\usepackage[ruled,noend,noline,slide]{algorithm2e}
\SetKwRepeat{Do}{do}{while}
\usepackage{algpseudocode}
\usepackage{csvsimple}
\usepackage{algcompatible}
\usepackage{amsfonts,mathtools}

\DeclareMathOperator{\off}{off}

\DeclareMathOperator{\sign}{sign}

\usepackage{listings}
\usepackage{subdepth} 
\newcommand{\B}[1]{\mbox{\boldmath $#1$}}
\title{A Jacobi-Type  Eigensolver for   Diagonally Dominant Symmetric  Matrices \thanks{The author is member of  INDAM/GNCS}}
\author{Luca Gemignani\footnotemark[2]}
\shorttitle{A Jacobi-Type  Iteration for   Diagonally Dominant Symmetric  Matrices} 
\shortauthor{L.~GEMIGNANI}
\begin{document}
\maketitle

\renewcommand{\thefootnote}{\fnsymbol{footnote}}

\footnotetext[2]{Dipartimento di Informatica, Universit\`a di Pisa, Largo B. Pontecorvo, 3  Pisa, Italy.}

\begin{abstract}
 This paper presents a Jacobi-type iteration for computing a  given specified eigenpair of a  symmetric  matrix.   For  a certain class of diagonally dominant matrices, the procedure is shown to converge at a linear rate depending on  how the matrix is significantly dominated.  The cost per iteration is  quadratic at most. 
Therefore,  the  proposed procedure  can compute an approximation of the desired eigenpair in quadratic time. 
\end{abstract}

\begin{keywords}
Eigenvalue computation,  Jacobi method, diagonal dominance 
\end{keywords}

\begin{AMS}
65F15, 65G05
\end{AMS}

\section{Introduction}
In this paper we are concerned with the numerical computation of a given specified 
eigenpair of a symmetric matrix  under  some diagonal dominance properties.  The problem is encountered in several different contexts.  

The Laplacian matrices of graphs are diagonally equipotent \cite{10.1016/j.automatica.2013.05.016}. Classes of generalized  diagonally dominant matrices play an important role in the theory of stability \cite{LIMEBEER01081982,572854}
 and have numerous applications in  economics, mathematical ecology, mechanics,  and other branches of science. 

Updating the eigendecomposition of a symmetric matrix under small or small-rank perturbations is an important  task in numerical linear algebra \cite{Golub,BN}. Let $A\in \mathbb R^{n\times n}$ be a symmetric matrix and $A=Q DQ^T$  its eigendecomposition, where $Q$ is orthogonal  and $D$ is diagonal.  Updating the eigendecomposition of $A+P$ amounts to compute 
the eigendecomposition of $B=D + Q^T P Q$. If $P$ is of small norm then $B$ is almost diagonal.  When $P$ is a rank-one matrix the perturbed eigenvalue problem is  associated with secular equation root-finding.   

The application of iterative root-finding schemes  to solve  the secular equation   provides a parallel  effective way to update the eigendecomposition of $D$ under  a rank-one perturbation. Many specialized root-finders have been designed in the literature (see \cite{Mel2} for a brief review and numerical comparisons). In particular, 
in \cite{Tilli1997} a continuation method was proposed to track the eigenvalue/zero paths simultaneously. At each step a problem with a small correction is solved.  However,  such procedures  may suffer from  the  classical weaknesses and limitations of algebraic root-finding methods in terms of both  robustness and conditning. Moreover,  they are limited to eigenvalue computation and have to be complemented with other different techniques to update the full set of eigenvectors. 

Traditional ways of  computing the eigenpairs of symmetric matrices using orthogonal transformations  include the Jacobi algorithm and the QR algorithm \cite{GOLUB200035}. Although the former  requires more arithmetic
computations,   it exhibits a large degree of  potential parallelism \cite{luk}. In addition, Jacobi-type  methods work very efficiently with almost diagonal matrices.  The  quadratic convergence of Jacobi methods applied to a  class of diagonally dominant matrices is analyzed in \cite{Matejas2000ER, Matejas2008ER}. Based on this analysis in  \cite{DRMAC2000191}  a way to accelerate the convergence of   the Jacobi algorithm is provided. Let $U$  be an  approximate eigenvector matrix of $A$, then applying the Jacobi algorithm on $B=U^TAU$  can speed up the reduction into diagonal form. Notice that   $B$  can likely be  a diagonally dominant matrix. 

In this work,  by elaborating on the results in \cite{Matejas2000ER, Matejas2008ER} we propose a  Jacobi-type method for computing  a given specified  eigenvalue and the corresponding eigenvector  of a diagonally dominant symmetric  matrix. If we are interested in a larger subset of all eigenpairs, then  our method is parallel in nature since all the approximations can be computed simultaneously and independently.  Therefore,  it provides a viable alternative to  simultaneous root-finding schemes  in the design of homotopy-based eigenvalue/eigenvector  algorithms, such as the one proposed  in \cite{Tilli1997}. 
For the class of  diagonally dominant  matrices  considered in \cite{Matejas2000ER, Matejas2008ER}  our proposed method exhibits a 
fast  linear convergence  which depends on  how the input matrix is significantly dominated.  The resulting  algorithm is numerically robust and accurate,  as it uses  orthogonal transformations only. The cost per iteration is quadratic at most. Therefore,  under conditions of reasonably fast convergence the  proposed procedure computes an approximation of the desired eigenvalue in quadratic time.
Numerical experiments  confirm the robustness and  efficiency of the proposed algorithm by showing that convergence can occur under weaker conditions  especially for the computation of well-separated eigenvalues of dense, symmetric matrices with diagonal dominance properties.

The paper is organized as follows. Section \ref{two} describes  a motivating example and introduces the notation used throughout the paper .  Section 
\ref{three}   presents  our  eigenvalue algorithm with  its convergence analysis.  Numerical experiments are illustrated in Section \ref{four}.  Finally, conclusions and future work are drawn in Section \ref{five}.

\section{ Motivating Example and Preliminaries}\label{two}
The Golub leukemia dataset \cite{Golub1999-zl} is a foundational microarray dataset used in bioinformatics for cancer class discovery and prediction, specifically distinguishing between Acute Lymphoblastic Leukemia (ALL) and Acute Myeloid Leukemia (AML). It is highly regarded as a benchmark for clustering algorithms in gene expression analysis.  In \cite{HIGHAM200725} this dataset is used  for testing spectral clustering techniques.  The leukemia dataset is stored  as a matrix $A\in \mathbb R^{7128\times 72}=(a_{i,j})$, where  $a_{i,j}$ measures  the activity of the $i$-th gene in the sample from the  $j$-th patient. The dataset consists of measurements on 72 leukemia patients, 47 "ALL"  and  25 "AML". The weight matrix $W = A^TA$   of size $n=72$ describes the similarities  among different patients.  The {\em Fiedler vector}  $\B v_2$ corresponds to the eigenvector associated with the second smallest eigenvalue of the {\em normalized Laplacian} $L=D^{-1/2}(D-W)D^{-1/2}$, where $D$ is a diagonal matrix with $d_i=\sum_{j=1}^{n} w_{i,j}$,  $1\leq i\leq 72$. The usual "spectral clustering"  approach is to divide the dataset  in two clusters based on the entries of the Fiedler vector \cite{HIGHAM200725}.  Since $L$ is symmetric the classical Jacobi eigenvalue algorithm can be employed  to calculate  the selected eigenpair $(\lambda_2, \B v_2)$. 
This algorithm applies a sequence of elementary plane rotations  in order to reduce the matrix $L$ into  diagonal form by  orthogonal similarity.  The cost is  generally $O(n^3)$. However, since we are interested in the second   eigenpair only,  one might consider  a simple modification of the  basic algorithm where all the rotations are restricted to act in the plane of coordinates  $(k, 2)$, $1\leq k\leq n$, $k\neq 2$.  If this change is successful, the modified  scheme would aim to reduce the computational cost  per iteration from $O(n^3)$ to $O(n^2)$   by highlighting an approximation of the second eigenvalue  $\lambda_2$ in position $(2,2)$.  In Figure  \ref{start} we illustrate the numerical results generated in MATLAB using this modification of the Jacobi algorithm. In particular, the plots  show the convergence of the absolute error $|L_{2,2}^{(j)}-\lambda_2|$  and of the Euclidean norm of the off-diagonal entries in the second row of  $L^{(j)}$ by using a logarithmic scale on the y-axis.  The errors are measured at the end of each  $j$-th iteration where $n-1$ plane rotations are applied.  It is  clearly seen that the errors exhibit a linear convergence behavior. 
\begin{figure}
\begin{minipage}[t]{.5\linewidth}
\vspace{0pt}
\centering
\includegraphics[width=2in]{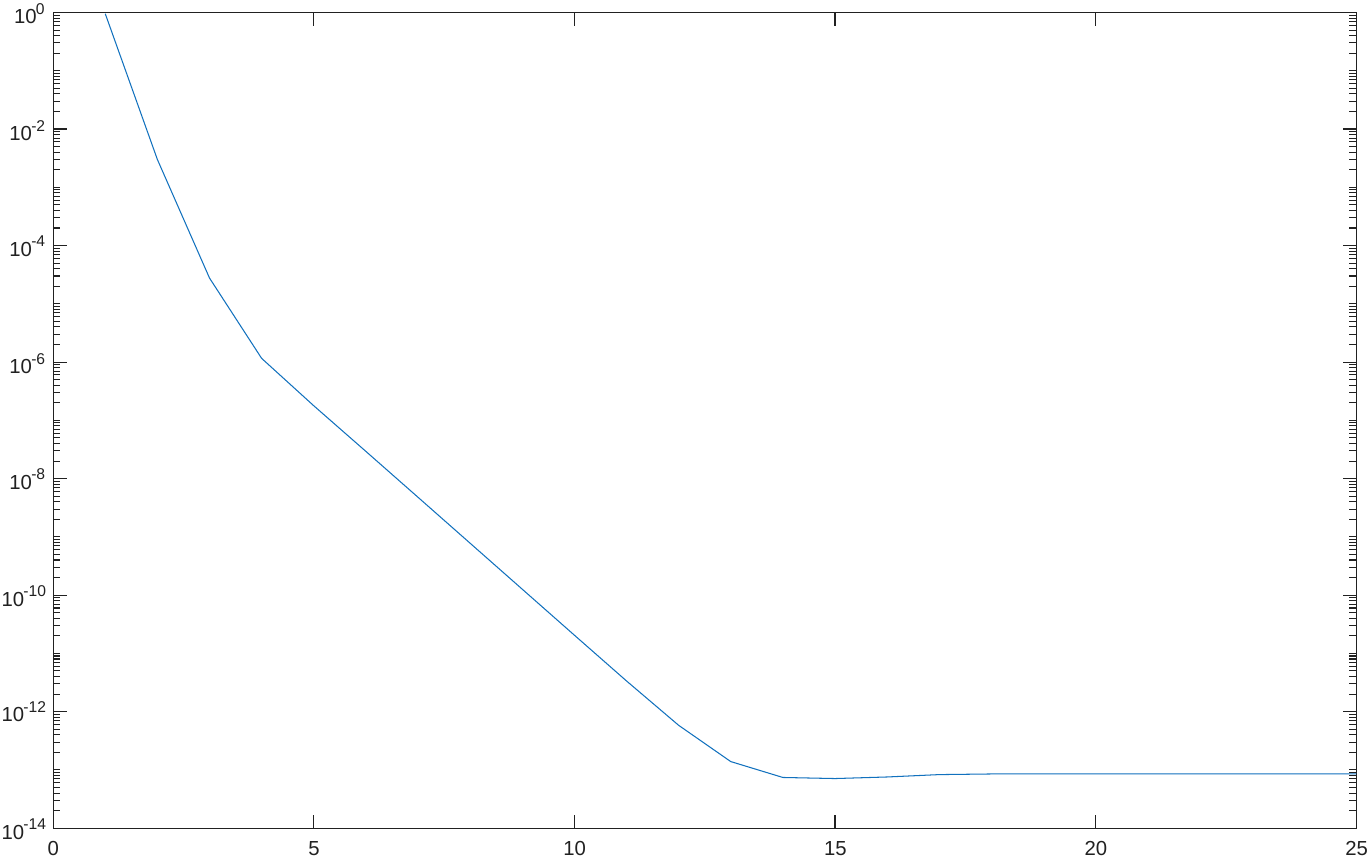}
\end{minipage}%
\begin{minipage}[t]{.5\linewidth}
\vspace{0pt}
\centering
\includegraphics[width=2in]{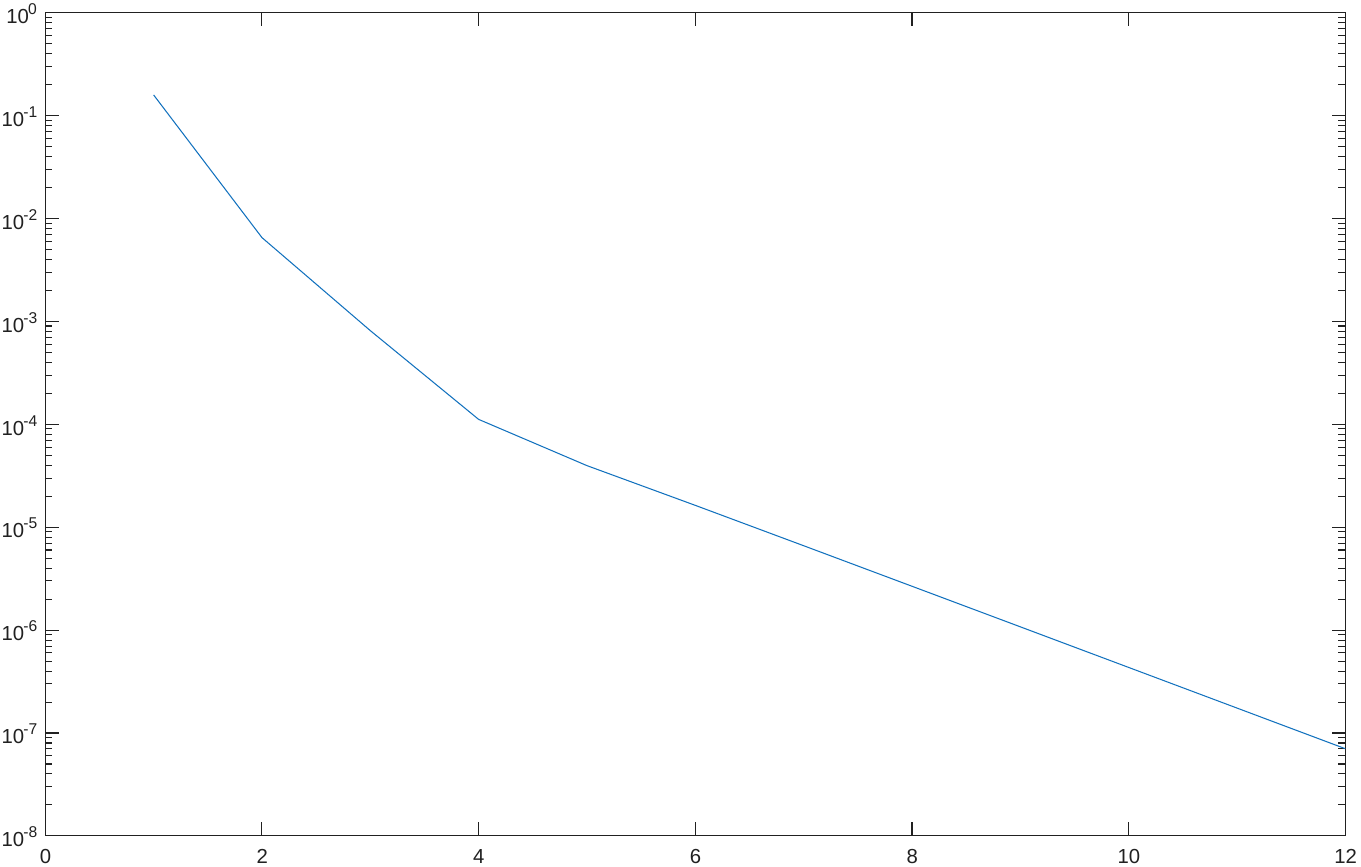}
\end{minipage}
\caption{Convergence history  of the quantities computed by our modification of the Jacobi algorithm applied to the normalized Laplacian of the leukemia dataset. The plot on the right shows the absolute error $|L_{2,2}^{(j)}-\lambda_2|$ whereas the plot on the left gives the Euclidean norm of the vector formed by the off-diagonal entries in the second row of  $L^{(j)}$ }\label{start}
\end{figure}
The primary motivation of this work is to explain the convergence of the modified eigenvalue algorithm.  

The customary Jacobi method  applied to   a symmetric  input matrix $A=(a_{i,j})\in \mathbb R^{n\times n}$
has the property of reducing the Frobenius norm of its  off-diagonal part. By $\Omega(A)$ we denote the off-diagonal part of $A$, i.e., $\Omega(A)=A-\diag(A)$.
Let us introduce the $\off$ operator given by 
\[
\off(A)=\parallel \Omega(A)\parallel_F=\sqrt{\parallel A\parallel_F^2 -\sum_{i=1}^n a_{i,i}^2} =\sqrt{
\sum_{i=1}^n\sum_{\substack{j=1\\ j\neq i}}^n a_{i,j}^2
}.
\]
Analogously, define 
\[
\off(A(k, \colon))=\sqrt{
\sum_{\substack{j=1\\ j\neq k}}^n a_{k,j}^2
}.
\]
Suppose to carry out  one step of the Jacobi algorithm $B\leftarrow Q A Q^T$, where $Q$ is a Givens rotation matrix that acts on the rows of indices $p$ and $q$, i.e., 
\[
Q=I_n + \left[ \B e_p, \B e_q \right ]\left[\begin{array}{cc} c-1&-s\\s & c-1
\end{array}\right] \left[ \B e_p, \B e_q \right ]^T
\]
so that 
\begin{equation}\label{ann}
\left[\begin{array}{cc} b_{p,p}&0\\0& b_{q,q}
\end{array}\right]=\left[\begin{array}{cc} c&-s\\s & c
\end{array}\right]\left[\begin{array}{cc} a_{p,p}& a_{p,q}\\a_{q,p} & a_{q,q}
\end{array}\right]\left[\begin{array}{cc} c&s\\-s & c
\end{array}\right].
\end{equation}
 From the invariance of the Frobenius norm under  the product by orthogonal matrices we deduce that 
 \[
 b_{p,p}^2 +  b_{q,q}^2 =a_{p,p}^2 +  a_{q,q}^2 + 2 a_{p,q}^2.
 \]
 Since $ \parallel A\parallel_F=\parallel B\parallel_F$ by using this relation  we find that 
 \begin{equation}\label{dec}
 \off(B)^2=\off(A)^2 -2 a_{p,q}^2.
 \end{equation}
 This property  is at the basis for  any convergence proof of the Jacobi eigenvalue algorithm.  The convergence rate  generally depends on the diagonal dominance properties of 
 the input  matrix $A$. 

Let us recall the notion of {\em scaled diagonally dominant} matrices from \cite{Barlow1990}.  Suppose that  $A=N+D\in \mathbb R^{n\times n}$ with  $D=\diag(A)$  and $D$  invertible. Under this assumption we can form the scaled matrix $H=|D|^{-1/2} A |D|^{-1/2}$.  Then,  $A$ is said a $\alpha$-scaled diagonally dominant matrix  w.r.t. the norm $\parallel \cdot \parallel$ if $\parallel \Omega(H) \parallel \leq \alpha$ for a  certain $\alpha$ with $0\leq \alpha<1$.   Observe that 
 \[
 H=|D|^{-1/2} A |D|^{-1/2}=|D|^{-1/2} N |D|^{-1/2} +\sign(D)
 \]
 and, therefore, 
 \[
 \parallel \Omega(H) \parallel =\parallel |D|^{-1/2} N |D|^{-1/2} \parallel.
 \]
If $A$ is symmetric diagonally dominant with positive diagonal entries, then 
\[
\parallel \Omega(H) \parallel_2 =\parallel D^{-1/2} N D^{-1/2} \parallel_2 <1
\]
since $D^{-1/2} N D^{-1/2}$ and $D^{-1} N$ are similar matrices.  Analogously, if $A$ is symmetric diagonally equipotent with positive diagonal entries then $\parallel \Omega(H) \parallel_2 \leq 1$. 

It can be shown that the class of scaled diagonally dominant matrices also includes  any symmetric positive definite matrix  which can be consistently ordered \cite{Barlow1990}. Other interesting properties which can be useful in our analysis are listed below. Specifically: 
\begin{description}
\item[Property 1.] The  relative distance between diagonal entries and eigenvalues can be estimated \cite{Hari}. 
Assume that $A$ is  a symmetric, $\alpha$-scaled diagonally dominant matrix  w.r.t. the  Euclidean norm   with distinct eigenvalues $\lambda_1< \lambda_2< \ldots < \lambda_n$ and its diagonal entries are ordered so that $a_{i,i}\leq a_{j,j}$ if $i<j$. Then, we obtain 
\begin{equation}\label{sep}
\frac{|a_{i,i}-\lambda_i|}{|a_{i,i}|}\leq\frac{4\cdot (\off(H(i,\colon)))^2}{\gamma}
\end{equation}
provided that $\alpha\leq  \gamma/(\gamma+3)$. Here  
$\gamma$ is the {\em minimum relative gap} in the spectrum of $A$ defined by 
\[
\gamma=\min_{1\leq j\leq n} \gamma_j, \quad \gamma_j=\min_{i\neq j} \frac{|\lambda_i-\lambda_j|}{|\lambda_i| + |\lambda_j|}.
\]
 Inequality \eqref{sep} can be extended to  symmetric matrices with multiple eigenvalues by considering the relative gap between disjoint sets of coalesced eigenvalues.
\item[Property 2.]  The orthogonal matrices  involved in the Jacobi process are also $\alpha$-scaled diagonally dominant.   Let $A$ be   a symmetric, $\alpha$-scaled diagonally dominant matrix  w.r.t. the  Frobenius norm.  Consider  a Givens plane rotation matrix determined  as in  \eqref{ann}  to annihilate the entries $a_{p,q}=a_{q,p}$.   There exist  suitable  positive constants $C_1<1$ and $C_2<1$ such that if $\alpha\leq C_1 \min\{\displaystyle \frac{1}{n}, \gamma\}$,  then  it follows (see Lemma 5 in  \cite{Matejas2008ER}):
\begin{equation}\label{small}
|\tan(\phi)|=\left|\frac{s}{c}\right|=\left|\frac{\sin(\phi)}{\cos(\phi)}\right| \leq 
\frac{C_2 |h_{p,q}|}{\gamma} \leq \frac{C_2 \alpha}{\gamma}.
\end{equation}
\end{description}

Property 1 is useful to guarantee that the possible convergence of $\off(H(i,\colon))$ towards zero implies the convergence of the $i$-th diagonal entry of $A$ to the desired eigenvalue.  Property 2 relates the relative magnitude of the off-diagonal entries of Givens rotation matrices with the ratio $\alpha/\gamma$
which is supposed to decrease under the Jacobi iteration.  This ratio will play a fundamental role in the convergence analysis of our proposed eigenvalue algorithm. 

\section{The Algorithm}\label{three}
Let $A=(a_{i,j})\in \mathbb R^{n\times n}$ be a  symmetric  matrix  with  
eigenvalues $\lambda_1\leq \lambda_\leq \leq \ldots \leq  \lambda_n$.  We  introduce the following {\bf Algorithm 1}   for approximating  the $m-$th eigenvalue $\lambda_m$.  The algorithm makes use of the  MATLAB function {\tt schur} which determines the Schur decomposition of a $2\times 2$ matrix. If $B\in \mathbb R^{2\times 2}$ is symmetric then $[U,T]={\tt schur}(B)$ returns an orthogonal matrix $U$ and a diagonal matrix $T$ such that $B=UTU^T$.
  \begin{algorithm}[tb]\caption{Given  a symmetric $A=(a_{i,j})\in \mathbb R^{n\times n}$, a tolerance threshold  $tol\in \mathbb R^+$ and  the integer $m\in [1,n]$ this procedure attempts to return  an  approximation $\widehat \lambda_m$ of the $m-$th eigenvalue $\lambda_m$  of $A$.}
    \begin{algorithmic}
    \STATE $[ignore, p]={\tt sort}(\diag(A))$; $A=A(p,p)$; 
     \WHILE {\tt until converged}
         \FOR  {$k=1\colon m-1$};
         \IF {$|A(m,k)|\geq tol$}
         \STATE $[U,T]={\tt schur}(A([k,m], [k,m])$; 
         \IF {$T(1,1)>T(2,2)$}
          \STATE $U=U(\colon, 2:-1:1)$; 
         \ENDIF
         \STATE $A([k, m], :)= U^T A([k, m], :)$;  $A(:, [k, m])= A(:, [k, m])U$;
         \ENDIF
          \ENDFOR
          \FOR  {$k=n :-1: m+1$};
          \IF {$|A(m,k)|\geq tol$}
         \STATE $[U,T]={\tt schur}(A([m, k], [m,k])$; 
         \IF {$T(1,1)>T(2,2)$}
          \STATE $U=U(\colon, 2:-1:1)$; 
         \ENDIF
         \STATE $A([m, k], :)= U^T A([m, k], :)$;  $A(:, [m, k])= A(:, [m, k])U$;
         \ENDIF
          \ENDFOR
 \ENDWHILE
 \STATE $\widehat \lambda_m=A(m,m)$
    \end{algorithmic}
\end{algorithm}
\begin{remark}
The first instruction is a preprocessing step. Reordering the entries of $A$ in accordance with the sorting of its diagonal entries can facilitate the convergence 
of the selected diagonal entry towards the corrected eigenvalue. 
\end{remark}
\begin{remark}
 In the classical Jacobi procedure, the  Givens rotation $U$  acting in the plane $(p,q)$  is determined  to annihilate   the entry $a_{p,q}$  by setting 
 \begin{equation}\label{jacformula}
 U=\left[\begin{array}{cc} \cos(\phi)&\sin(\phi)\\-\sin(\phi) & \cos(\phi)
\end{array}\right], \quad \tan(2 \phi)= \frac{2 a_{p,q}}{a_{q,q}-a_{p,p}}, \quad \phi\in [-\pi/4, \pi/4].
 \end{equation}
 For this choice of $U$ it can easily be  seen that the ordering of the diagonal entries in position $p$ and $q$ is invariant under the transformation.  Specifically, if we compute 
 \[
A([p, q], :)= U^T A([p, q], :);  \quad A(:, [p, q])= A(:, [p,q])U;
 \]
 then we find 
 \[
 a_{p,p} \leftarrow  a_{p,p}-\frac{\sin(\phi)}{\cos(\phi)} a_{p,q}, \quad a_{q,q} \leftarrow  a_{q,q}+\frac{\sin(\phi)}{\cos(\phi)} a_{p,q}.
 \]
 If $a_{q,q}>a_{p,p}$  holds before the transformation then from \eqref{jacformula} $\sign(\tan(\phi))=\sign(a_{p,q})$ and therefore the ordering of the diagonal entries is maintained under the transformation.  We conclude similarly  if otherwise the condition $a_{q,q}<a_{p,p}$  is satisfied. 
 Thus, the conditional statements in  {\bf Algorithm 1} ensure that the  ordering produced by Schur decompositions is in accordance with the one generated by means of the  classical procedure.  We have preferred the current formulation to  stress its significance w.r.t the ordering of the diagonal entries of $A$.
\end{remark}

{\bf Algorithm 1} requires $O(n^2)$ arithmetic operations per iteration at most. 
In the next subsection we investigate  its  convergence properties. 

\subsection{Convergence Analysis}
 The proposed algorithm generates a sequence of matrices  
 $\{A^{(k)}\}_{k\in \mathbb N}$  with $A^{(0)}=A\in \mathbb R^{n\times n}$ and $A^{(k+1)}\leftarrow Q^{(k+1)}A^{(k)} {Q^{(k+1)}}^T$ for  $k\geq 0$. The orthogonal matrix $Q^{(k+1)}=I_n + \left[ \B e_p, \B e_q \right ]U^T \left[ \B e_p, \B e_q \right ]^T$  is used at step $k+1$  of the algorithm.  Each iteration requires $n-1$ orthogonal transformations ( (possibly equal to identity if no action is necessary). The subsequence $\{A^{( k(n-1))}\}_{k\in \mathbb N}$ is formed from the matrices generated at the end of each cycle in the {\tt while} command.

 Next result proves the convergence of the proposed annihilation scheme for a positive tolerance threshold.
 \begin{theorem}\label{genconv}
 Let  $\{A^{(k)}\}_{k\in \mathbb N}$  be the sequence of matrices generated by {\bf Algorithm 1} applied to the symmetric matrix  $A^{(0)}=A$  with  a given fixed tolerance threshold $tol>0$.  Then,  we have $\lim_{k\rightarrow +\infty}\off(A^{(k)}(m, \colon)=\gamma$ with $\gamma< \sqrt{n} \cdot tol$
 \end{theorem}
 \begin{proof}
 Notice that  $\off(A^{(k)})\geq 0$ is bounded   and from   \eqref{dec} it  is monotonically nonincreasing.  Therefore,  it admits a limit which is the infimum of the sequence. 
If this infimum is zero  then the same limit holds for $\off(A^{(k)}(m, \colon)$.
If, otherwise, the limit is nonzero,  say $\ell>0$,   we have that $\forall \epsilon>0$
 $\exists M\in \mathbb N$ such that $\forall k\geq M$ it holds $\ell^2 \leq \off(A^{(k)})^2\leq \ell^2 +\epsilon$. Suppose that  $\epsilon\leq tol^2$ and that there exists $j\neq m$ such that $|a_{m,j}^{(M)}|\geq tol$. Then,  from \eqref{dec} we obtain 
\[
\off(A^{(M+1)})^2\leq \ell^2 +\epsilon -2 \  tol^2\leq \ell^2-\epsilon
\]
which contradicts the infimum definition.  Therefore, we conclude that  $|a_{m,j}^{(k)}|< tol$
 for all $k\geq M$ and $j\neq m$.
 \end{proof}
 
 We illustrate pictorially in Figure \ref{fnew} the convergence  of our proposed eigenvalue algorithm applied to a  symmetric random matrix of order $n=129$ for the approximation of  its  eigenvalue $\lambda_m$ with $m=(n+1)/2$.  
\begin{figure}
\begin{minipage}[t]{.5\linewidth}
\vspace{0pt}
\centering
\includegraphics[width=2in]{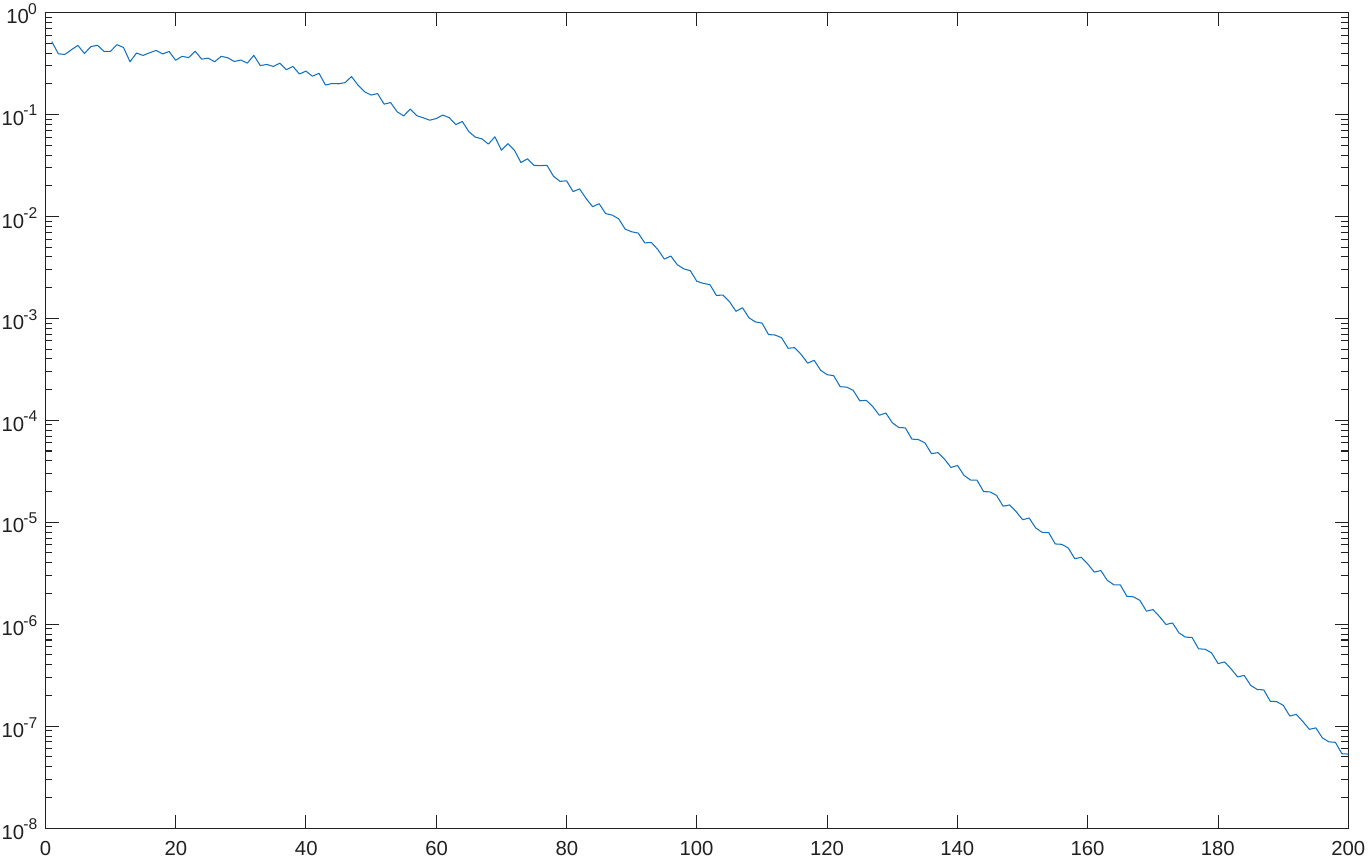}
\end{minipage}%
\begin{minipage}[t]{.5\linewidth}
\vspace{0pt}
\centering
\includegraphics[width=2in]{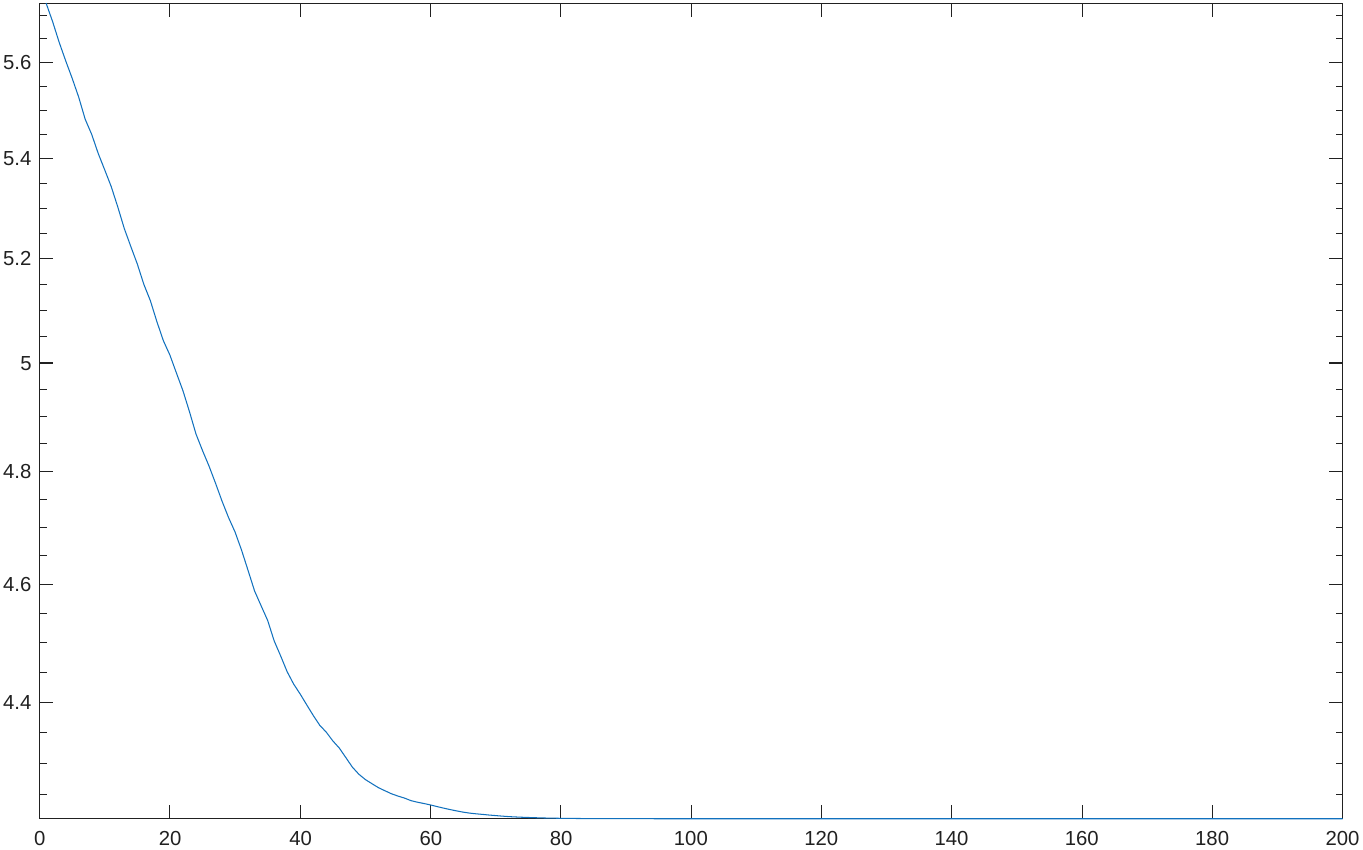}
\end{minipage}
\caption{Convergence history for  Procedure {\textsc{EIG\_DDSM}}. The figure on the left shows  the convergence  of $\off(A^{(k)}(m, \colon))$  whereas the figure on the right hand side gives the monotonic decreasing of $\off(A^{(k)})$.}\label{fnew}
\end{figure}
However,  in general,  the convergence of the iterative algorithm  
can be  excruciatingly slow. Moreover, the result stated in  Theorem \ref{genconv} does not  generally guarantee that $a_{m,m}^{(k)}$ is an approximation of the  sought $m-$th eigenvalue of $A$. The exploitation of the diagonal dominance property makes crucial to obtain reliable bounds as well as to prove the  correctness of our proposed algorithm.


  Recall that $\{A^{(k)}\}_{0\leq k \leq M}$ with $M=\ell(n-1)$   is the  finite sequence  of matrices generated by {\bf Algorithm 1} in the first  $\ell$ cycles of the {\tt while} command.  We would like  to estimate the reduction of  $\off(A^{(k)}(m, \colon))$ for this  sequence of matrices when  $tol=0$.  Let  $H^{(k)}=|D_k|^{-1/2} A^{(k)} |D_k|^{-1/2}$,  $D_k=\diag(A^{(k)})$,  denote the  scaled iterate  whenever it can be formed.
 For the sake of notational  simplicity,  we also set  
 \[
 \alpha_k= \off(H^{(k)})=\parallel \Omega(H^{(k)})\parallel_F. 
 \]
The following result  adjusts  the general convergence analysis  for the  Jacobi method  performed in \cite{Matejas2008ER} to our algorithm.   Some  constants  might be improved, but
since it is clear that these improved constants are not tight  anyway,  we prefer to present a shorter
proof.

\begin{theorem}\label{convth}
  Let $A=(a_{i,j})\in \mathbb R^{n\times n}$, $n\geq 3$,  be  a symmetric matrix such that $a_{1,1}\leq a_{2,2}\leq \ldots \leq a_{n,n}$, $a_{i,i}\neq 0$, $1\leq i\leq n$,  and, moreover, $\alpha_0\leq \displaystyle\frac{1}{11} \min\{\frac{1}{n},\gamma\}$. For the matrices generated by {\bf Algorithm 1}  applied to $A=A^{(0)}$ with $tol=0$ in the first $n$ cycles of the while statement  the following properties hold:
  \begin{enumerate}
   \item  $\alpha_k\leq 1.001  \ \alpha_0, \quad  0\leq k\leq n(n-1)$; 
  \item $\off(H^{((n-1)\ell)}(m, \colon)) \leq  \displaystyle\left(\frac{2.8\cdot 1.001 \ \alpha_0}{\gamma}\right)^\ell \alpha_0,  \quad 1\leq \ell \leq n$.
  \end{enumerate}
\end{theorem}
 \begin{proof}
From Lemma 3 in   \cite{Matejas2008ER} it follows that 
\[
\alpha_{k}\leq \sqrt{\left( 1+\frac{0.00126}{n^2}\right)^{n(n-1)}}\alpha_{0} \quad \quad  0\leq k\leq n(n-1)
\]
which implies 
\[
\alpha_{k}\leq \sqrt{e^{0.00126}} \ \alpha_{0}\leq 1.001  \ \alpha_{0}.  
\]
Now  recall that   the Jacobi method can be applied under several permutation equivalent orderings. In particular,  one  cycle of  {\bf Algorithm 1}  generates the same matrix as the one obtained after the first $n-1$ Jacobi steps in the reversed column-wise ordering applied to a permuted matrix where the  diagonal entry in position $m$ is moved to the last position. 
Observe that 
\[
\alpha_k \leq 1.001  \ \alpha_{0}\leq \frac{1.001}{11} \min \{\frac{1}{n},\gamma\} \leq\frac{1}{10}\min \{\frac{1}{n},\gamma\}.
\]
From this  by virtue of Theorem 6 in \cite{Matejas2008ER},  we deduce that 
\[
 \off(H^{((n-1)\ell)}(m, \colon)) \leq  \frac{2.8 \alpha_{(n-1)(\ell-1)}}{\gamma}\off(H^{((n-1)(\ell-1))}(m, \colon)),  \quad 1\leq \ell \leq n.
\]
Thus,  we  conclude that 
\[
 \off(H^{((n-1)\ell)}(m, \colon)) \leq  \left(\frac{2.8\cdot 1.001 \ \alpha_0}{\gamma}\right)^\ell \alpha_0,  \quad 1\leq \ell \leq n.
\]
 \end{proof}
 
Under  the hypotheses of Theorem \ref{convth},  $\alpha_k\leq \gamma/(\gamma+3)$ and therefore
 from \eqref{sep} it follows that $a_{m,m}^{(k)}$ approaches the $m-$th eigenvalue $\lambda_m$ of $A$. 
While Theorem \ref{convth}  predicts  under quite restrictive assumptions a linear convergence of the iterates with a convergence rate less than or equal to $\varsigma$, 
 \[
 \varsigma=\frac{2.8\cdot 1.001 \ \alpha_0}{\gamma} \leq  \frac{2.8\cdot 1.001 }{11}\leq 0.26, 
 \]
the actual
convergence of the iterates  can be faster and take place under more favourable conditions.
The following example illustrates the dynamic of convergence of Jacobi process on a small matrix. Experiments with larger matrices will be discussed in the next section. 

\begin{example}\label{exx1}
  Let $A\in \mathbb R^{11\times 11}$ be the matrix given by 
  \[
  a_{i,j}=\left\{\begin{array}{lll} i \ {\rm if} \ i=j; \\
  \frac{1}{121}  \ {\rm if} \ (i,j \neq 6) \land  (i\neq j); \\[5pt] 
  \frac{1}{100} \ {\rm if} \ (i\neq j=6) \lor (6=i\neq j).
  \end{array}\right.
  \]
Numerically we find that  $\gamma=0.048$, $\gamma_6=0.076$, $\alpha_0=0.023$ and $\alpha_0/\gamma\simeq 0.48$. The condition $\alpha_0\leq  C\displaystyle\min \{\frac{1}{n},\gamma\}$ is  satisfied with a larger constant $C\simeq 1/2$ than $1/11$.  However, our proposed algorithm exhibits a  fast linear convergence.  In Figure  \ref{fig1} we show  the plot of $\off(H^{(10(\ell-1)}(6, \colon))$ using a base $10$ logarithmic scale for the y-axis together with a  table of its first $6$ values. By using the  function {\tt polyfit} from MATLAB, the estimated rate of convergence is  $\widehat \varsigma= \num{5.48e-3}$.  
\begin{figure}
\begin{minipage}[t]{.55\linewidth}
\vspace{0pt}
\centering
\includegraphics[width=2in]{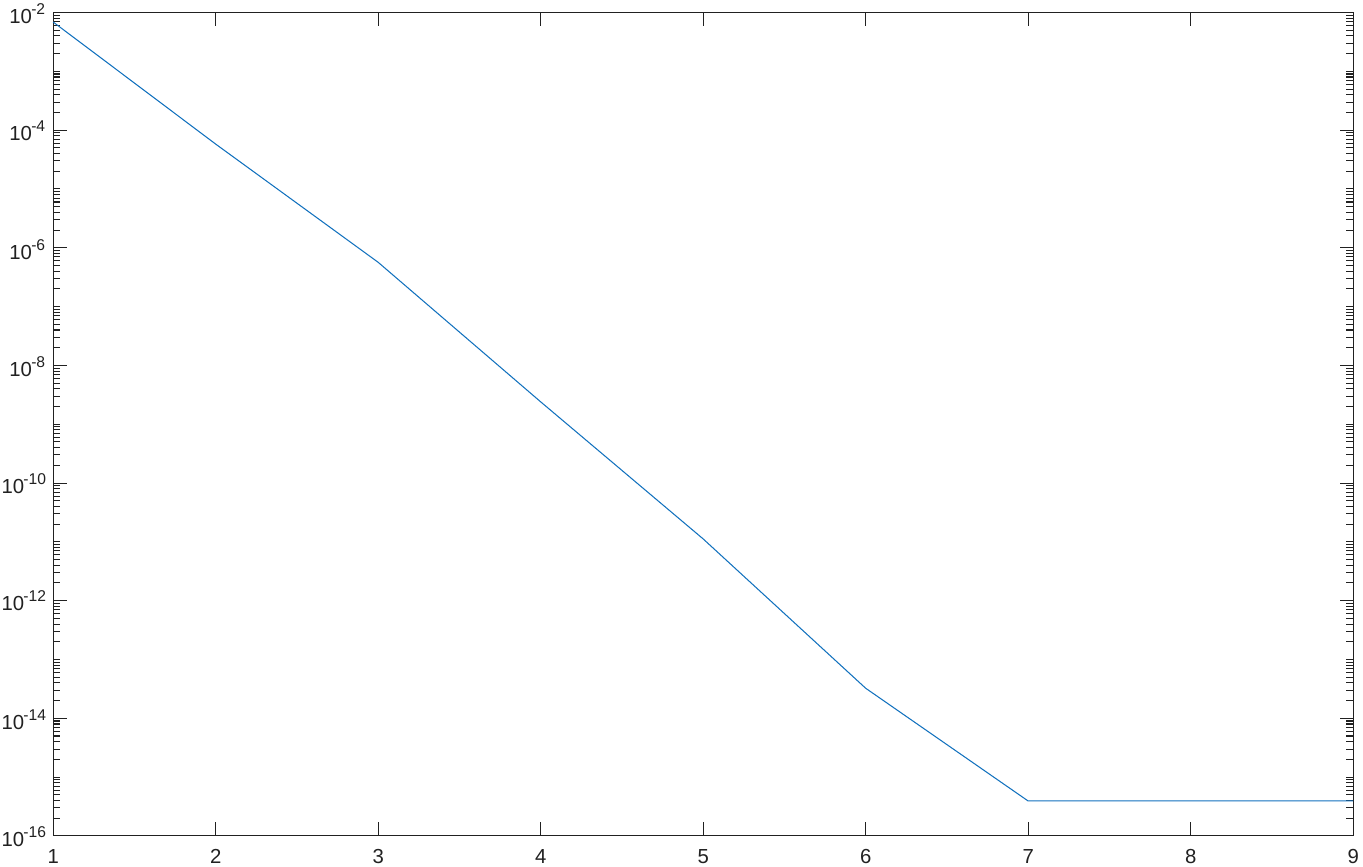}
\end{minipage}%
\begin{minipage}[t]{.35\linewidth}
\vspace{0pt}
\centering
\begin{tabular}{| c | c |}
        \hline
        $\ell$ &  $\off(H^{(10(\ell-1)}(6, \colon))$\\
        \hline
        1 & \num{6.89e-03} \\
        2 & \num{5.81e-05}\\
        3 &  \num{5.67e-07}\\
         4 & \num{2.41e-09}\\
          5 & \num{1.12e-11} \\
           6 & \num{3.25e-14} \\
        \hline
      \end{tabular}
\end{minipage}
\caption{Reduction of $\off(H^{(10(\ell-1)}(6, \colon))$  in the first steps of  {\bf Algorithm 1} applied to the matrix $A$ of Example \ref{exx1} with $m=6$. }\label{fig1}
\end{figure}
\end{example}
For the leukemia dataset we find $\gamma_2=\num{3.45e-03}$ and $\alpha_0\simeq 1$ and, therefore, the condition $\alpha_0\leq  C\displaystyle\min \{\frac{1}{n},\gamma\}$ is not fulfilled for any constant $C<1$.  
In order to   understand the results in Figure \ref{start} we note that  verifiable  conditions and error bounds can be obtained without spectral information by means of a first order error analysis.  
\begin{lemma}[\cite{Matejas2000ER}]\label{new}
Consider  a Givens plane rotation matrix determined  as in  \eqref{jacformula}  to annihilate the entries $a_{p,q}=a_{q,p}$ of $A=\left[\begin{array}{cc} a_{p,p}& a_{p,q}\\a_{q,p} & a_{q,q}
\end{array}\right]$ with $a_{p,p}\neq a_{q,q}$, $a_{p,p}, a_{q,q}\neq 0$. Then 
\[
|\tan(\phi)|\leq \frac{1}{2} \frac{|h_{p,q}|}{{\tt rel}(a_{p,p}, a_{q,q})}, \quad 
{\tt rel}(a_{p,p}, a_{q,q}) =\frac{|a_{p,p}-a_{q,q}|}{|a_{p,p}| + |a_{q,q}|}.
\]
\end{lemma}
\begin{proof}
 From   \eqref{jacformula} we have 
 \begin{equation}\label{bbq}
 |\tan(\phi)|\leq \frac{1}{2}|\tan( 2\phi)| \leq \frac{ 
 \displaystyle\frac{|a_{p,q}|}{\sqrt{|a_{p,p} a_{q,q}|}}}{\displaystyle\frac{|a_{p,p}-a_{q,q}|}{\sqrt{|a_{p,p} a_{q,q}|}}}.
 \end{equation}
 Hence, the thesis follows from the  the inequality of arithmetic and geometric means which implies 
 \[
 \displaystyle\frac{|a_{p,p}-a_{q,q}|}{\sqrt{|a_{p,p} a_{q,q}|}} \geq 
 2 {\tt rel}(a_{p,p}, a_{q,q}).
 \]
\end{proof}

Next result gives a first order analysis to describe the reduction of $\off(H^{((k)}(m, \colon))$  in {\bf Algorithm 1} during the execution of a cycle of the {\tt while} command.  Without loss of generality,  we  may assume $m=n$. 

\begin{theorem}\label{foa}
 Let  $A=A^{(0)}=(a_{i,j}^{(0)})\in \mathbb R^{n\times n}$, $n\geq 3$,  be  a symmetric matrix such that $0<a_{1,1}^{(0)}<a_{2,2}^{(0)}< \ldots < a_{n,n}^{(0)}$.  For the matrices generated by  {\bf Algorithm 1} applied to $A=A^{(0)}$  with $tol=0$ and $m=n$ define:
 \vskip-0.4cm
 \begin{align*}
 \widehat  \gamma &= \min_{1\leq p\leq n-1} \widehat \gamma_p, \quad \widehat \gamma_p=\displaystyle|\frac{a_{n,n}^{(0)}}{a_{p,p}^{(0)}}-1|;  \\
 \epsilon_k &=  \off(H^{(k)}(1:n-1, n)), \quad  \widehat  \alpha_k= \off(H^{(k)}(1\colon n-1, 1\colon n-1)).
 \end{align*}
 \vskip-0.1cm
 \noindent
 Then, in first order error analysis   neglecting  the $O(\epsilon_0^2)$ and $O(\epsilon_0 \widehat \alpha_0^2)$ terms, we find:
 \begin{enumerate}
     \item $\widehat \alpha_{n-1}\overset{\cdot}{=}\widehat \alpha_{0}$;
     \item $\epsilon_{n-1} \overset{\cdot}{\leq}\displaystyle\frac{\widehat \alpha_0}{\sqrt{2}\widehat\gamma} \epsilon_0 $.
 \end{enumerate}
 \end{theorem}
\begin{proof}
    First,  notice that from the trigonometric identity $\cos^2(\phi) +\sin^2(\phi) =1$  under the assumption that $\tan(\phi)=O(\epsilon_0)$ it follows that $\cos(\phi)\overset{\cdot}{=}1$. Therefore, without restriction one  may consider diagonally scaled versions of Givens rotation matrices.  It is easily found that 
    \[
    a_{j,n}^{(j-1)}=a_{j,n}^{(0)} + \sum_{k=1}^{j-1}\tan(\phi^{(k)}) a_{j,k}^{(0)}, \quad 1\leq j\leq n-1.
    \]
    From this, by virtue of Lemma \ref{new}  the proof proceeds  by induction on the iteration number $j$, $1\leq j\leq n-1$.  Specifically,  it can be shown   that 
    \[
    a_{p, p}^{(j)} \overset{\cdot}{=}  a_{p, p}^{(0)}, \quad 1\leq p\leq n-1, \quad 1\leq j\leq n-1, 
    \]
    and 
    \[
    |h_{p,q}^{(j)}|\overset{\cdot}{=}|h_{p,q}^{(0)}|, \quad 1\leq p,q\leq n-1, \quad 1\leq j\leq n-1.
    \]
   Moreover, from \eqref{bbq} one deduces that  
   \[
   |\displaystyle\frac{\tan(\phi^{(k)}) a_{j,k}^{(0)}}{\sqrt{a_{j,j}^{(0)}a_{n,n}^{(0)}}}| \overset{\cdot}\leq 
   \displaystyle\frac{|h_{k,n}^{(0)}| |h_{j,k}^{(0)}|}{\widehat \gamma_k}, \quad 1\leq k\leq j-1, 
   \]
and, hence, 
    \[
        |h_{j,n}^{(j-1)}|\overset{\cdot}{\leq }  |h_{j,n}^{(0)}| + \sum_{k=1}^{j-1} \displaystyle\frac{1}{\widehat \gamma_k}|h_{k,n}^{(0)}| |h_{j,k}^{(0)}|, \quad 1\leq j\leq n-1.
    \]
    Since  we have 
    \[
    a_{j,n}^{(n-1)}\overset{\cdot}{=}\sum_{k=j+1}^{n-1}\tan(\phi^{(k)})a_{j,k}^{(0)}
    \]
    then, we obtain that 
    \begin{equation}\label{dym}
     |h_{j,n}^{(n-1)}|  \overset{\cdot}{\leq}
    \sum_{k=j+1}^{n-1} \displaystyle\frac{1}{\widehat \gamma_k}|h_{k,n}^{(0)}| |h_{j,k}^{(0)}|, \quad 1\leq j\leq n-1.
    \end{equation}
    Finally, by using the Cauchy-Schwarz inequality for the standard 2-norm, this implies that 
    \[
    \epsilon_{n-1} \overset{\cdot}{\leq}\frac{\widehat\alpha_0}{  \sqrt{2} \widehat \gamma} \epsilon_0. 
    \]
\end{proof}

This theorem and in particular equation \eqref{dym} describes the dynamics of entry generation in the $m-$th column (row).
For the leukemia dataset, after the first two iterations it is found that  $\epsilon_0 \widehat \alpha_0^2\simeq\num{1.0e-07}$,  with $\widehat  \alpha_0= \off(H^{(2(n-1))}(3\colon n, 3\colon n))$ and $\epsilon_0 =  \off(H^{(2(n-1))}(:, 2))$, and 
$\displaystyle\frac{\widehat \alpha_0}{\sqrt{2}\widehat \gamma} \simeq 0.8$. Then, from  \eqref{dym}   we deduce the  linear convergence to zero of $\off(H^{(k)}(3:n, 2))$ and, hence, of $h^{(k)}_{1,2}$.

 \section{Numerical Results}\label{four}
 We have implemented  {\bf Algorithm 1} in MATLAB and then  we have performed numerical experiments  to confirm the  effectiveness of our proposed eigenvalue algorithm.   For the sake of visualization,   the tolerance is set to 0 and the iteration is stopped if $\off(A(m, \colon))\leq\parallel A\parallel_F \sqrt{\epsilon}$, where $\epsilon$ is the machine precision and $n$ is the size of the input matrix 
 $A$.  Our implementation also returns as output the parameters 
\[
\gamma_m=\min_{j\neq m} \frac{|\lambda_m-\lambda_j|}{|\lambda_m| + |\lambda_j|},  
\]
and 
\[
\alpha_0=\off(|\diag(A)|^{-1/2} A|\diag(A)|^{-1/2}).
\]
We  have conducted  numerical experiments in three different scenarios to validate the proposed approach. 

 \begin{enumerate}

\item The first set of experiments is aimed to further  illustrate the  application of our proposed method for spectral graph partitioning. The wine-quality dataset  (see https://archive.ics.uci.edu/dataset/186/wine+quality) is another 
 classical example used for testing classification algorithms.  The task is  to predict wine quality based on the physicochemical features. The dataset includes 
 4898 samples  of different wines with 11  corresponding attributes.  Hence, the clustering problem consists of $n=4898$ points $\B x_i$ in $\mathbb R^{11}$.  Following the algorithm  proposed in \cite{NIPS2001_801272ee}, the matrix $W=(a_{i,j})\in \mathbb R^{n\times n}$ has entries 
 \[
 w_{i,j}={\tt exp}(-\parallel \B x_i-\B x_j\parallel_2/(2\sigma^2)), \quad i\neq j; \quad w_{i,i}=0.
 \]
 Let $A=L=D^{-1/2}(D-W)D^{-1/2}$ be the normalized Laplacian with $D=\diag\left[d_1, \ldots, d_n\right]$, $d_i=\sum_{j=1}^n w_{i,j}$.
 In Figure \ref{fnew1} we show  the plots of  $err_j=|L^{(n-1)(j-1)}(m,m) -\lambda_m|$ and $\off(L^{((n-1)(j-1))}(m, \colon))$, $j\geq 1$,  for  $m\in \{2,3\}$ and $\sigma=10$. We have $\gamma_2=\num{4.5e-2}$, $\gamma_3=\num{9.4e-3}$ and $\alpha_0\simeq 1$. After two iterations with $m=2$ we find $\min_{3\leq k\leq n} \widehat \gamma_k=\num{1.1e-1}$ and  $\off(H^{(2(n-1))}(3\colon n, 3\colon n))=\num{4.7e-2}$
 After three iterations with $m=3$ we obtain $\min_{4\leq k\leq n} \widehat \gamma_k=\num{3.5e-2}$ and   $\off(H^{(3(n-1))}(4\colon n, 4\colon n))=\num{3.0e-2}$.
\begin{figure}
\begin{minipage}[t]{.55\linewidth}
\vspace{0pt}
\centering
\includegraphics[width=2in]{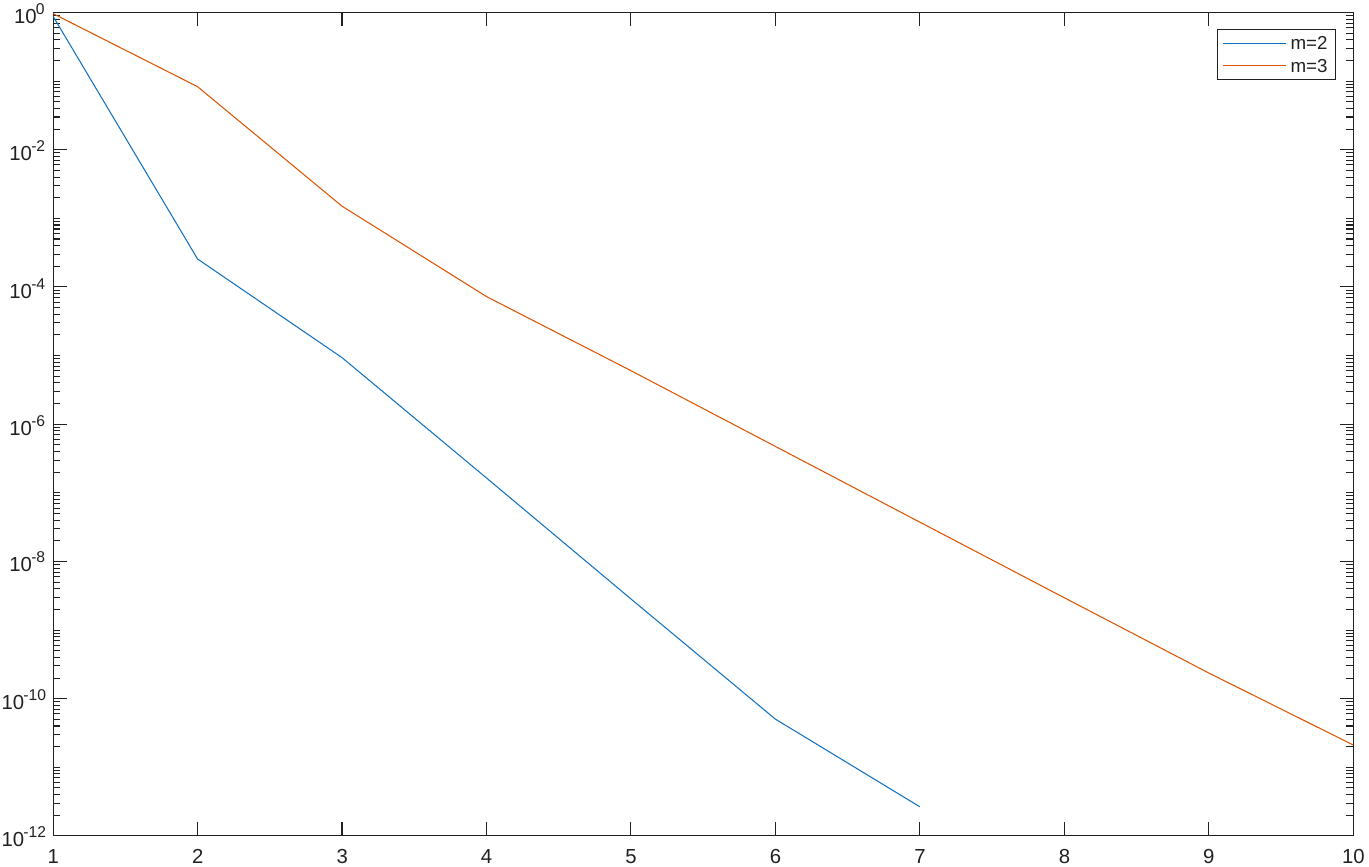}
\end{minipage}%
\begin{minipage}[t]{.35\linewidth}
\vspace{0pt}
\centering
\includegraphics[width=2in]{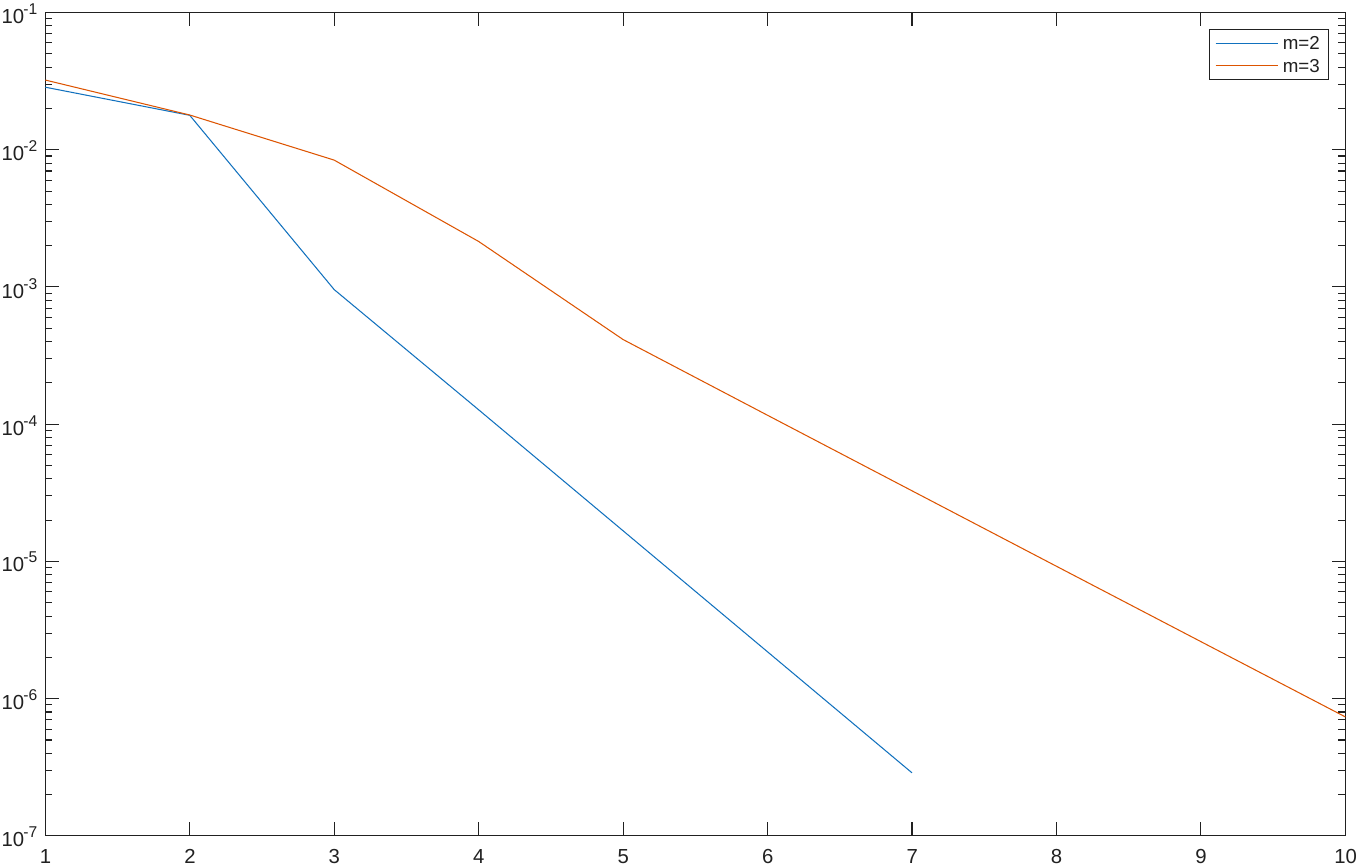}
\end{minipage}
\caption{Convergence history of {\bf Algorithm 1} applied for the approximation of the second and  the third eigenvalue of the  normalized Laplacian  generated from the wine-quality dataset.}\label{fnew1}
\end{figure}
 Notice that fast convergence can be expected  even for moderately large values of $\rho=\alpha_0/\gamma_m$. Similar behaviors  are also  observed with adjacency matrices generated from random undirected graphs.  In Figure \ref{fnew2} we consider  the application  of {\bf Algorithm 1} for the approximation of the second smallest eigenvalue  of the normalized Laplacian  of size $n=2000$ associated with the adjacency matrix $A$ of  the undirected random graph G22 from the SuiteSparse Matrix Collection (formerly the University of Florida Sparse Matrix Collection).  For this problem  our algorithm returns $\alpha_0\simeq 10$ and $\gamma_2=\num{4.4e-2}$.
 \begin{figure}
\begin{minipage}[t]{.55\linewidth}
\vspace{0pt}
\centering
\includegraphics[width=2in]{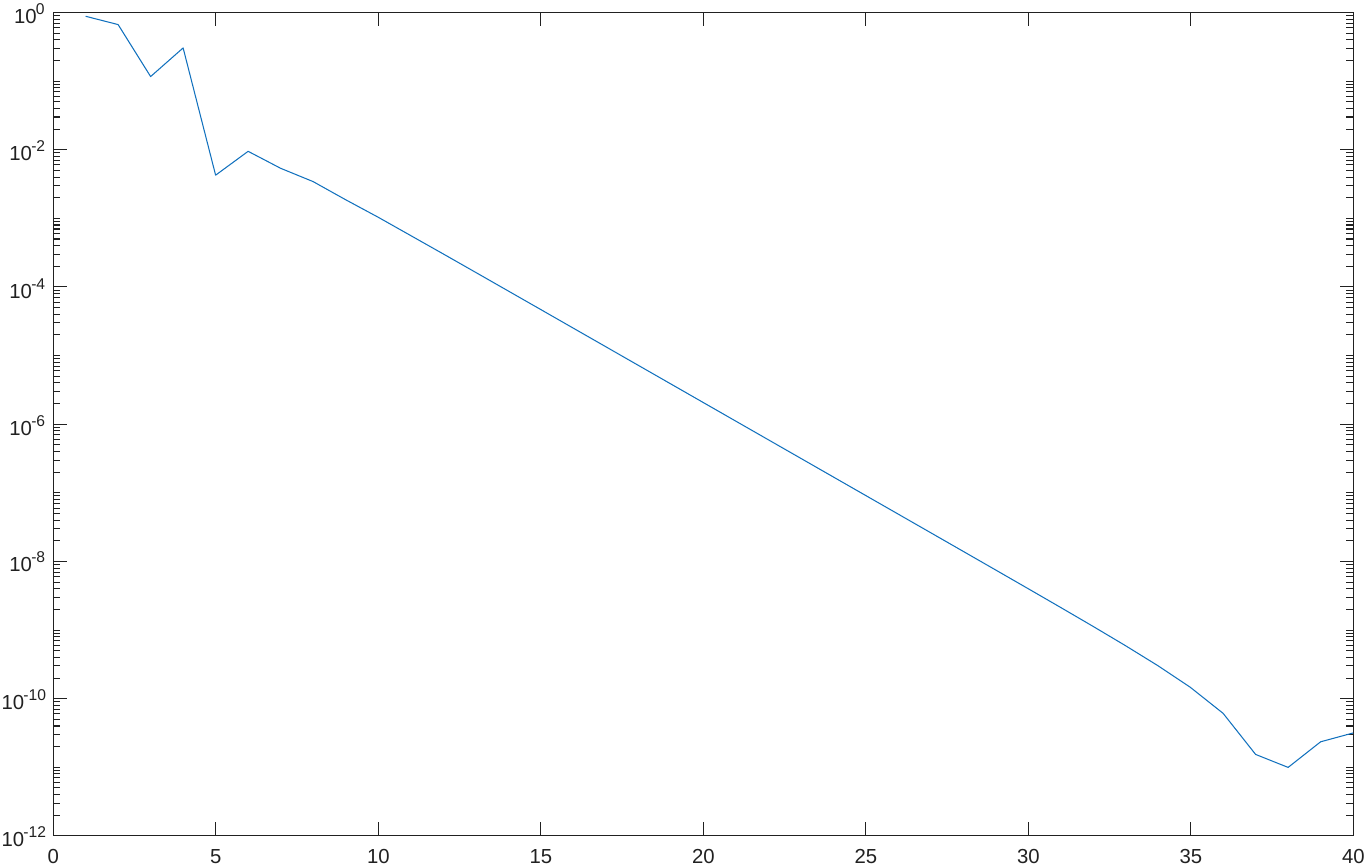}
\end{minipage}%
\begin{minipage}[t]{.35\linewidth}
\vspace{0pt}
\centering
\includegraphics[width=2in]{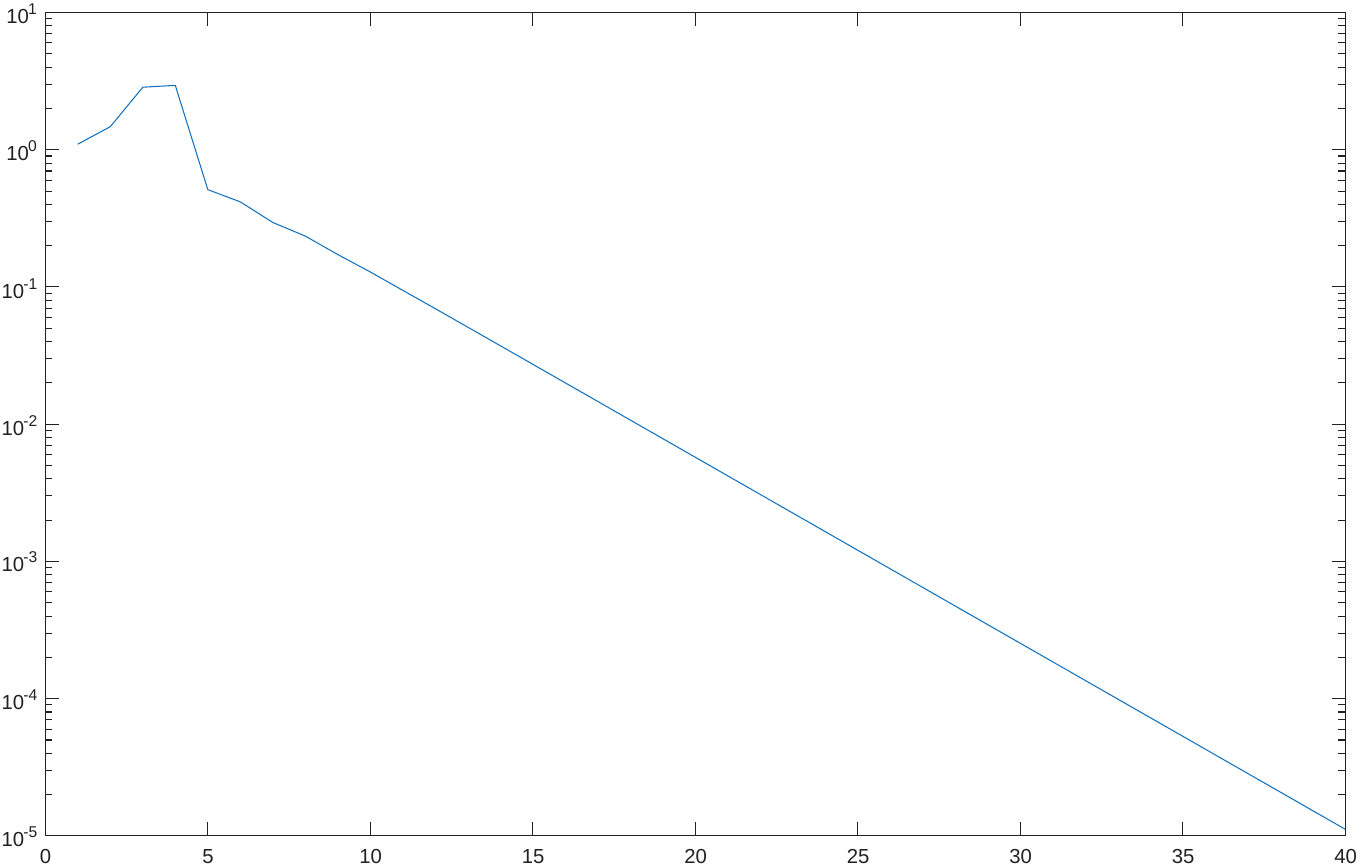}
\end{minipage}
\caption{Convergence history of {\bf Algorithm 1} applied for the approximation of the second eigenvalue of the  normalized Laplacian  generated from  G22. The plots  on the left  and on the right show the convergence of $err_j$ and $\off(L^{((n-1)(j-1))}(2, \colon))$, respectively.
}\label{fnew2}
\end{figure}
In Figure \ref{spy8} we illustrate the sparsity pattern of  the adjacency matrix $A$ of G22 together with the pattern of the  matrix suitably reordered according to  the ordering of the entries of the second eigenvector calculated by our proposed method. 
\begin{figure}
\begin{minipage}[t]{.55\linewidth}
\vspace{0pt}
\centering
\includegraphics[width=2in]{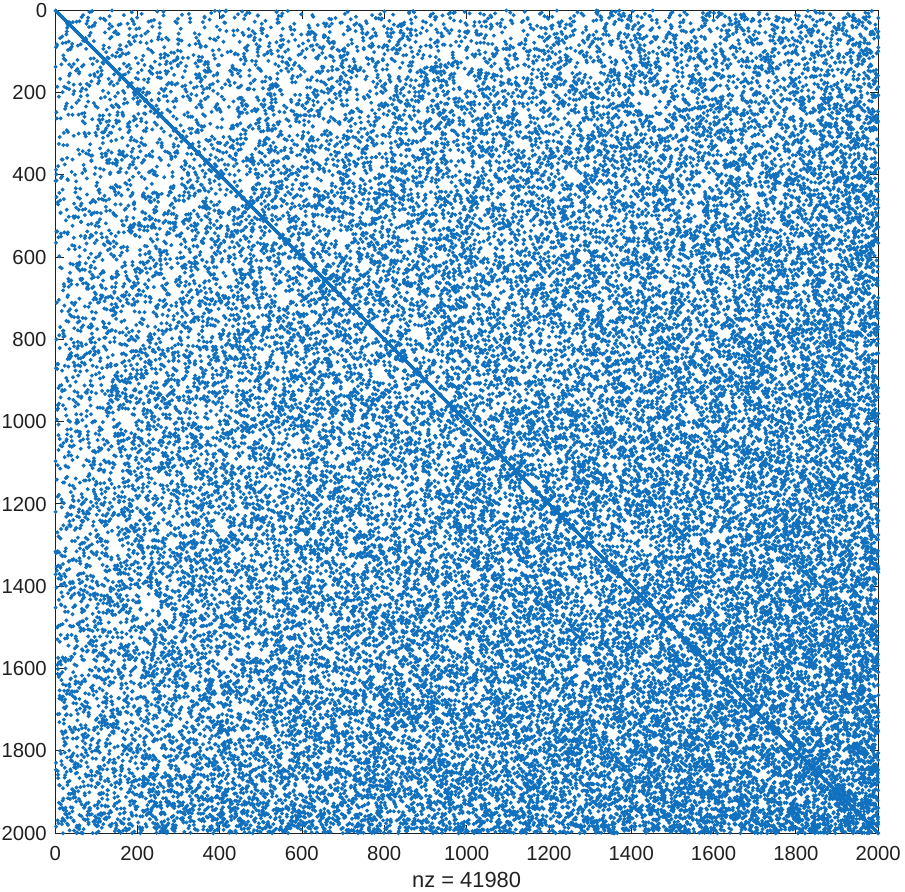}
\end{minipage}%
\begin{minipage}[t]{.35\linewidth}
\vspace{0pt}
\centering
\includegraphics[width=2in]{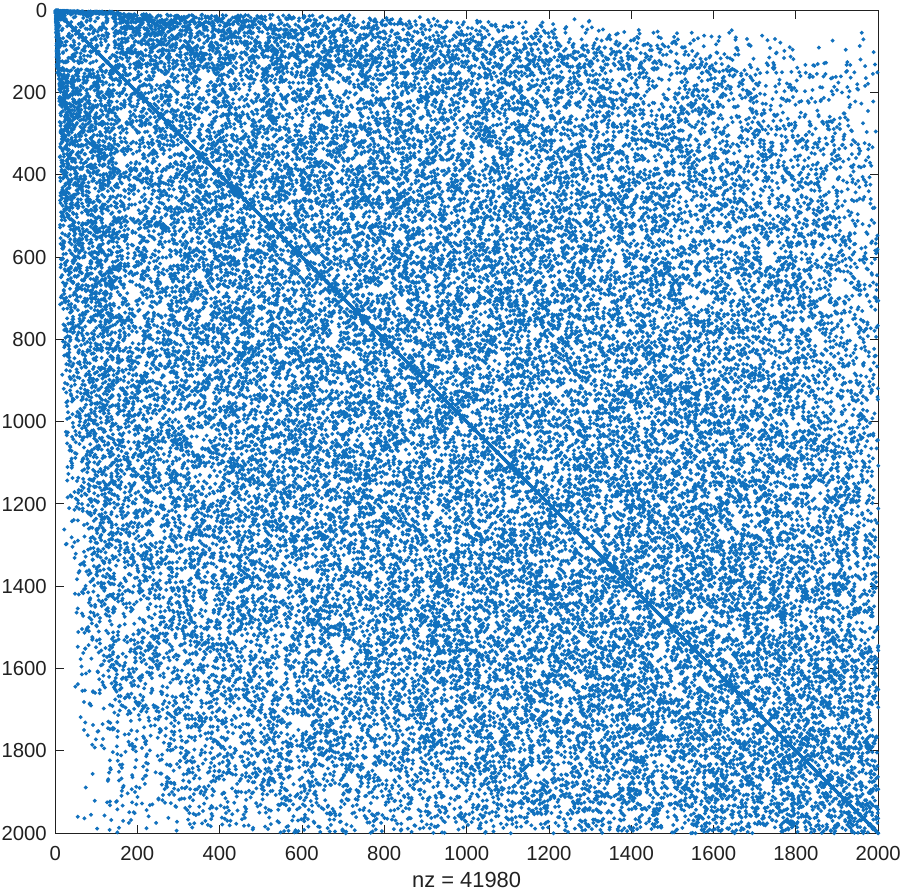}
\end{minipage}
\caption{Sparsity patterns of the adjacency matrix   generated from  G22 and its permuted version according to the ordering of the entries of the second eigenvector of the normalized Laplacian.}\label{spy8}
\end{figure}
For comparison we have tested the convergence of  the MATLAB function {\tt eigs} on these problems. In both  cases the convergence is very fast in two or three iterations but,  to our knowledge {\tt eigs} performs a preliminary factorization at the cost of $O(n^3)$.  

However, the application  of {\bf Algorithm 1} to  symmetric equipotent matrices  deserves further investigation since algorithm performance looks to be  problem dependent.  In particular,  it generally performs poorly when 
it is applied to  large sparse graphs since  slow convergence coupled with "fill-in".      Problem 442 from the SuiteSparse Matrix Collection  is a  sparse Laplacian matrix $A=D-W$ of order $n=10774$.  In Figure \ref{figult} we illustrate the convergence of our proposed method for the approximation of the  smallest eigenvalue of the  normalized Laplacian  generated from $A$ and   from $A+\diag\left[1\colon n\right]$. Observe that this latter matrix is diagonally dominant. The plots show  the errors  $err_j=|L^{(n-1)(j-1)}(1,1) -\lambda_1|$ and $errn_j=\off(L^{((n-1)(j-1))}(1, \colon))$, $j\geq 1$.  The flat behavior of the curves for the matrix $A$ indicates that   {\bf Algorithm 1} performs inefficiently. 
\begin{figure}
\begin{minipage}[t]{.55\linewidth}
\vspace{0pt}
\centering
\includegraphics[width=2in]{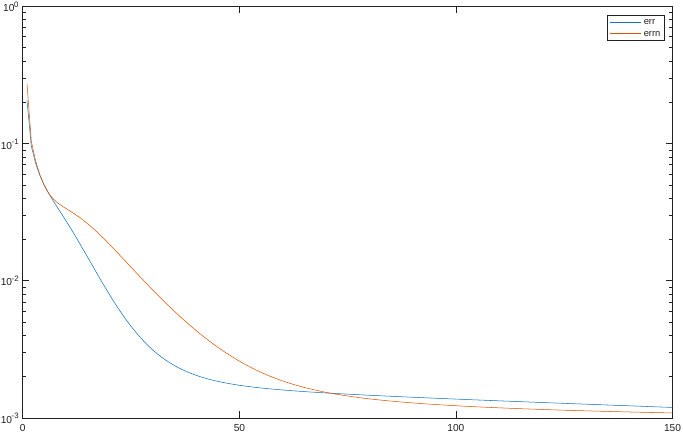}
\end{minipage}%
\begin{minipage}[t]{.35\linewidth}
\vspace{0pt}
\centering
\includegraphics[width=2in]{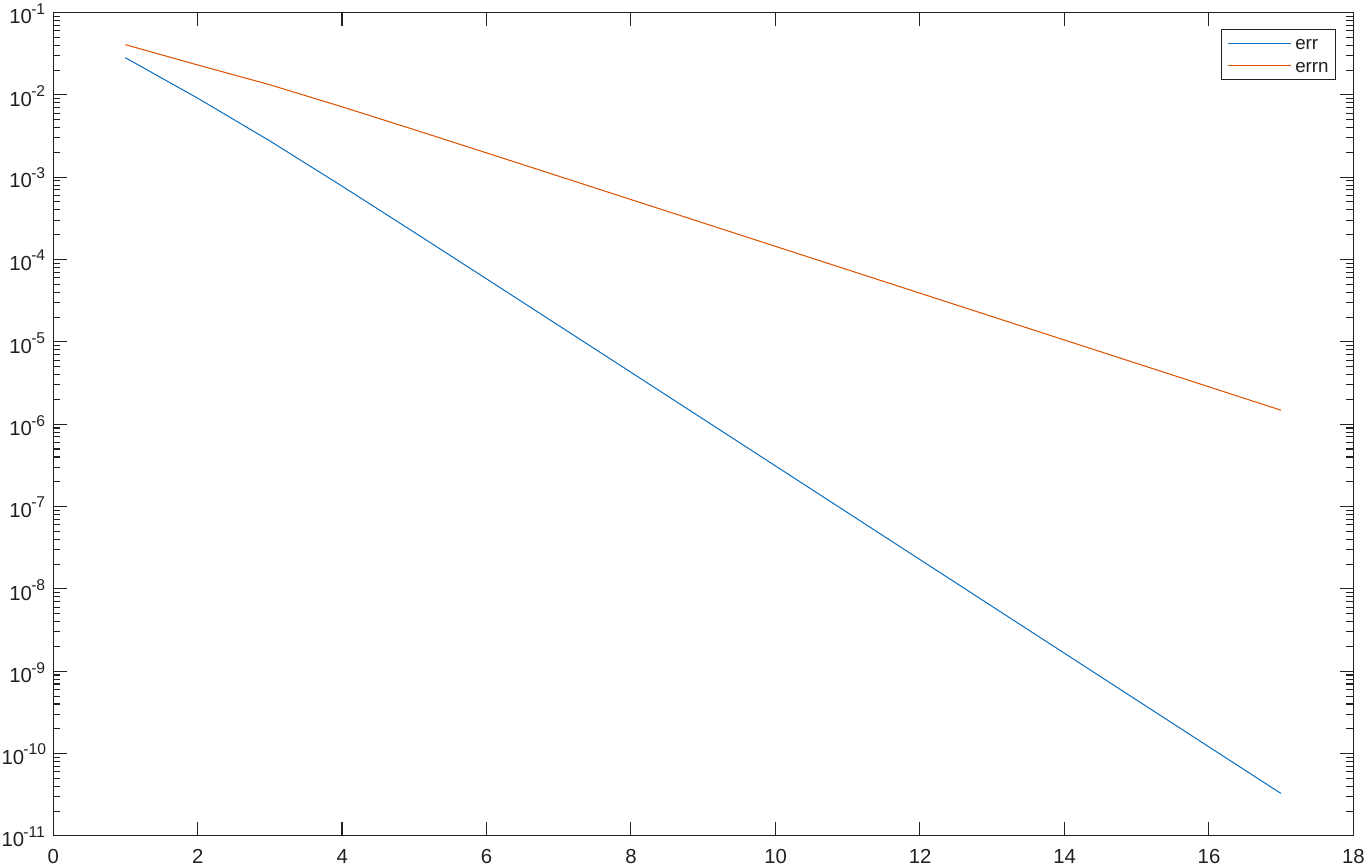}
\end{minipage}
\caption{Convergence history of {\bf Algorithm 1} applied for the approximation of the smallest eigenvalue of matrices generated from Problem 422.  The plot on the left show the errors for the normalized Laplacian. The  plot on the right gives  the errors for the modified matrix.}\label{figult}
\end{figure}

 \item  In the  second set of experiments we consider  diagonal plus rank-one matrices  with size $n$ of the form
 \begin{equation}\label{drk1}
 A=\diag[1+x_1, \ldots, 1+x_n] +\sigma \B u \B u^T
\end{equation}
 where $\B u^T=\left[\begin{array}{ccc} \sin(\sqrt{2}\pi x_1)& \ldots \sin( \sqrt{2}\pi x_n)
\end{array}\right]$,  $x_i=ih$, $h=1/(n+1)$, and $\sigma=1/n$. Observe that $A$ is positive definite and diagonally dominant and, therefore, it is diagonally scaled. 
In Figure \ref{fnew3} we illustrate the convergence history of {\bf Algorithm 1} applied to the matrix $A$ of size $n=1023$ for the approximation of the $m$-th eigenvalue with $m\in \{1, \displaystyle\frac{n+1}{2}, n\}$. 
 \begin{figure}
\begin{minipage}[t]{.55\linewidth}
\vspace{0pt}
\centering
\includegraphics[width=2in]{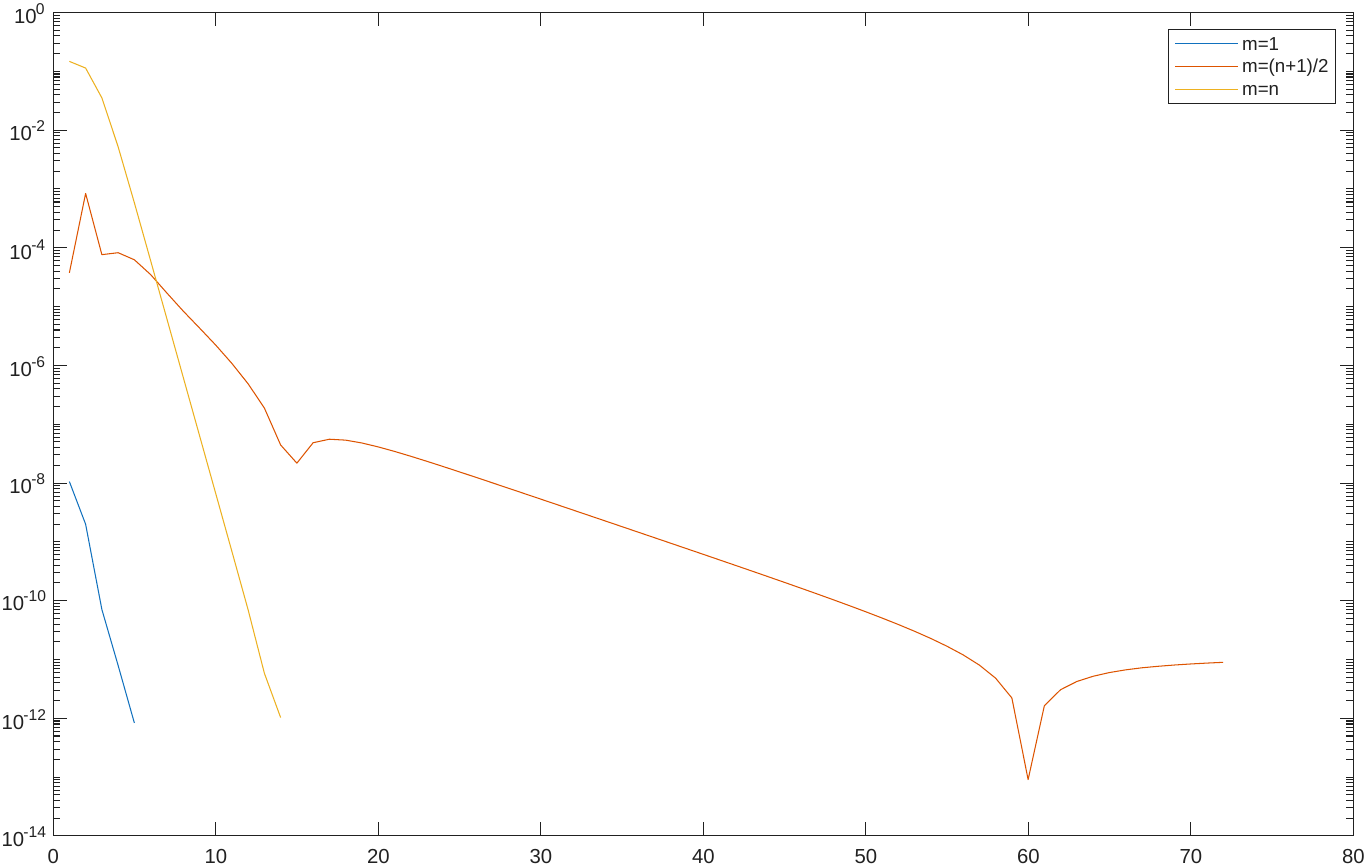}
\end{minipage}%
\begin{minipage}[t]{.35\linewidth}
\vspace{0pt}
\centering
\includegraphics[width=2in]{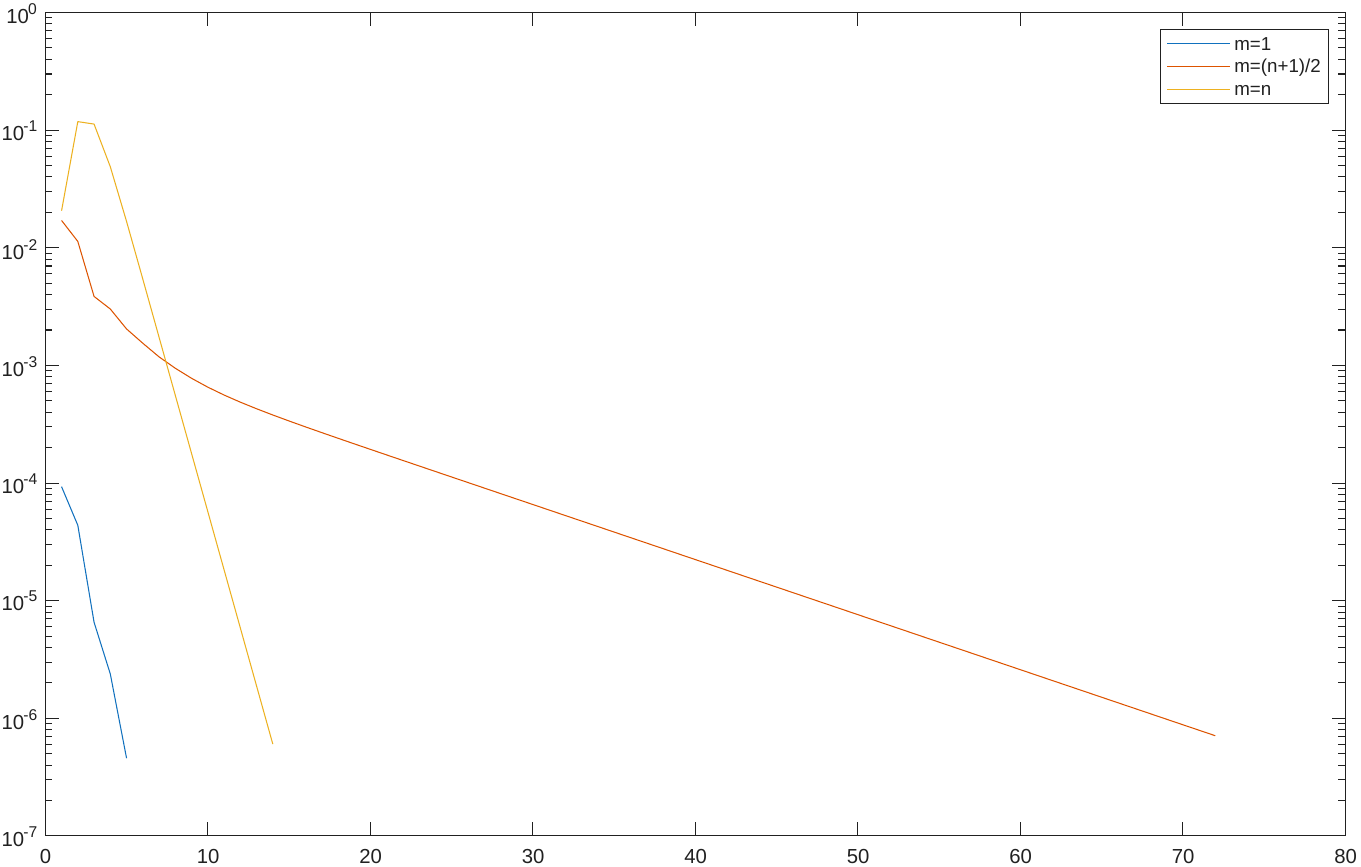}
\end{minipage}
\caption{Convergence history of {\bf Algorithm 1} applied to the matrix \eqref{drk1}.  The plots  on the left  and on the right shows the convergence of $err_j$ and $\off(A^{((n-1)(j-1))}(m, \colon))$, respectively.
}\label{fnew3}
\end{figure}
It is worth noting the different behaviors. For $m=1$ we find $\alpha_0\simeq 0.3$ and $\gamma_1=\num{4.8e-4}$. The  initial error is quite small and the algorithm proceeds very quickly with the refinement.  For $m=\displaystyle\frac{n+1}{2}$,   $\gamma_m\simeq \gamma_1$ but the initial error is larger and, hence, the convergence takes more iterations.  Finally, for $m=n$ the initial error is quite large but 
$\gamma_n=\num{3.6e-2}$ and the convergence is much faster.  It also can be noticed that 
whenever the convergence is fast  the matrices generated in our procedure are rank-structured. There is a  wide literature starting with \cite {EGG} concerning the exploitation of the rank structure for eigenvalue computation of diagonal plus small-rank  matrices.  The adaptation of {\bf Algorithm 1} for working with a generator-based representation of  a diagonal plus small-rank  matrix is an ongoing research. 

\item In the third  set of experiments we have tested the  application of {\bf Algorithm 1} as an eigenvalue/eigenvector tracker in the spirit of the algorithm proposed in \cite{Tilli1997}.  Given a symmetric matrix $A\in \mathbb R^{n\times n}$ we consider the homotopy $A(t)=\diag(A) + t (A-\diag(A))$ with $t\in [0,1]$ by tracking the eigenvalue/eigenvector paths from $t=0$ to $t=1$.  The  step length  at time $t_k$ is  determined adaptively by  using the current available decomposition of $A(t_k)=Q(t_k) \Sigma(t_k){Q(t_k)}^T $.  In particular, we compute $\widehat \gamma_k=\min_{1\leq i\leq n} \gamma_i^{(k)}$ and then we determine  the step length $s_k$ such that $t_{k+1}=t_k +s_k$ to satisfy $ s_k \parallel A-\diag(A)\parallel_F/\widehat \gamma_k$ of order of unity.  Notice that 
\[
{Q(t_k)}^T A(t_{k+1}){Q(t_k)}=\Sigma(t_k) + s_k {Q(t_k)}^T(A-\diag(A)) {Q(t_k)}.  
\]
In Figure  \ref{figotto}
we show  the plots of eigenvalue paths and of  step length  generated by a random matrix of order $n=33$.  The number of steps is $137$ and the average number of iterations  per step and  eigenvalue of our proposed eigenvalue solver  is $4.8$. 
\begin{figure}
\begin{minipage}[t]{.55\linewidth}
\vspace{0pt}
\centering
\includegraphics[width=2in]{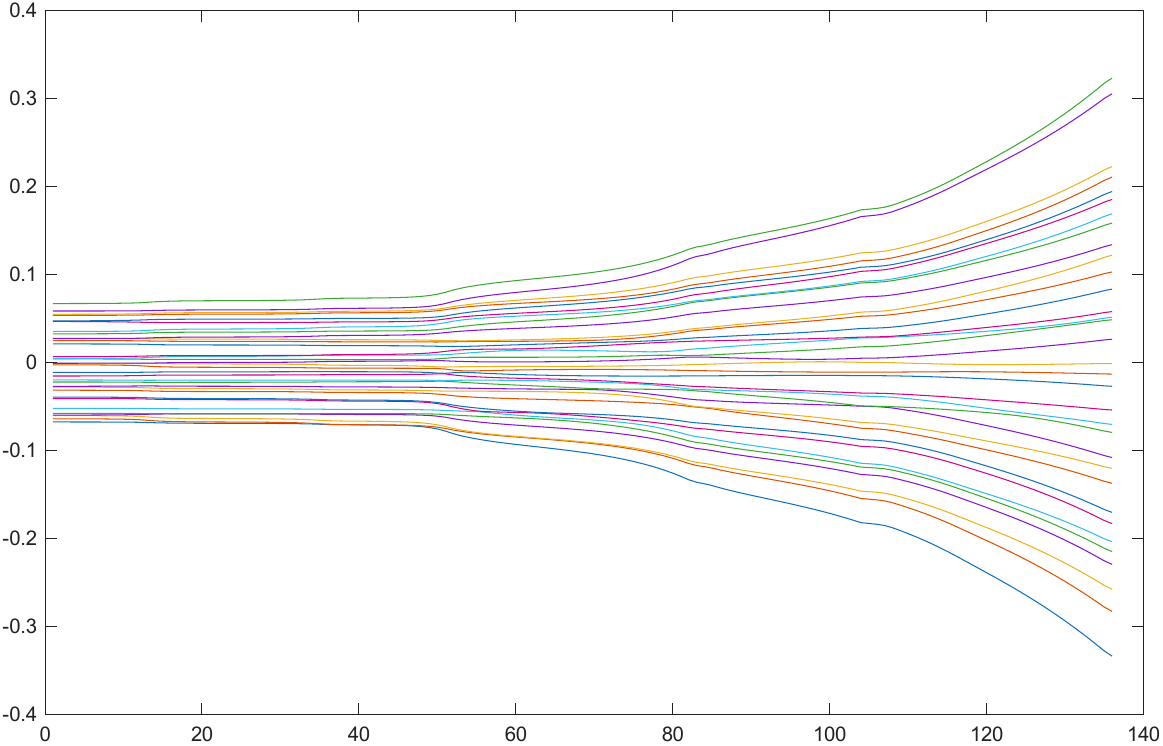}
\end{minipage}%
\begin{minipage}[t]{.35\linewidth}
\vspace{0pt}
\centering
\includegraphics[width=2in]{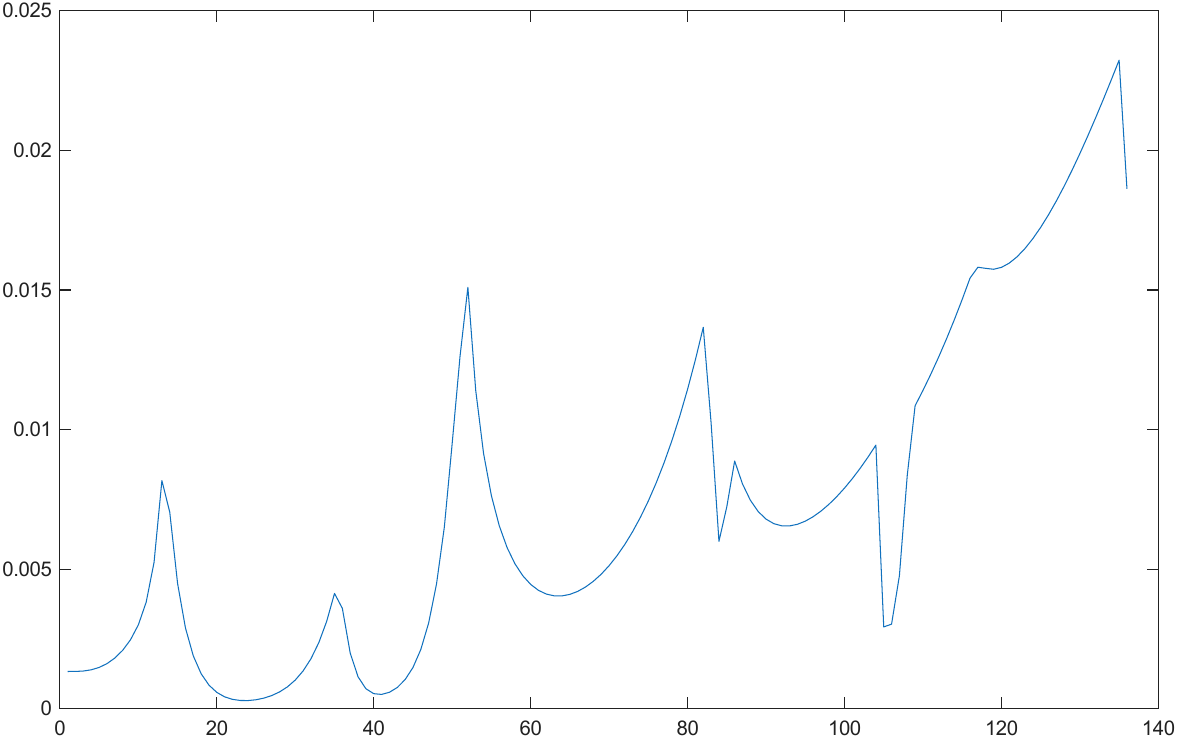}
\end{minipage}%
\caption{Plot of eigenvalue paths and  of the step length function for a random matrix $A$ of order $n=33$. }\label{figotto}
\end{figure}
\end{enumerate}

\section{Conclusion and Future Work}\label{five}
In this paper we have presented an adaptation of the Jacobi method for computing a given specified eigenpair of a symmetric matrix. The algorithm can  be of interest for the computation of well-separated eigenvalues of dense, symmetric matrices with diagonal dominance properties.
A convergence result has been provided under rather restrictive diagonal dominance assumptions. The numerical results   supported by a first order error analysis indicate that the  proposed algorithm exhibits  regular fast linear convergence under weaker conditions. In addition, difficult cases can possibly be detected in the early iterations.
Future work involves further investigation of convergence analysis to explain these numerical results. We are also interested in extending our algorithm to  approximate  a few selected eigenpairs that are well separated from the rest of the spectrum.

 \bibliographystyle{plain} 
\bibliography{secular}
\end{document}